\documentclass[10pt, reqno]{amsart}
\usepackage{amsfonts,amssymb,amsthm,amsmath,amscd,enumerate,enumitem, esint}

\makeatletter
\@namedef{subjclassname@2020}{%
	\textup{2020} Mathematics Subject Classification}
\makeatother

\usepackage[utf8]{inputenc}

\newtheorem{thm}{Theorem}[section]
\newtheorem{cor}[thm]{Corollary}
\newtheorem{lem}[thm]{Lemma}
\newtheorem{prop}[thm]{Proposition}
\theoremstyle{definition}
\newtheorem{defn}[thm]{Definition}
\theoremstyle{remark}
\newtheorem{rem}[thm]{Remark}

\numberwithin{equation}{section}
\allowdisplaybreaks
\begin{document}
	
	\title[New Gaussian Riesz transforms on variable Lebesgue spaces]{New Gaussian Riesz transforms on variable Lebesgue spaces}
	
	%Information for  first author
	\author[E. Dalmasso]{Estefan\'ia Dalmasso}
	\address{Estefan\'ia Dalmasso\newline
		Instituto de Matem\'atica Aplicada del Litoral, UNL, CONICET, FIQ.\newline Colectora Ruta Nac. N$^\circ$ 168, Paraje El Pozo,\newline S3007ABA, Santa Fe, Argentina}
	\email{edalmasso@santafe-conicet.gov.ar}

	% Information for second author
	\author[R. Scotto]{Roberto Scotto}
	\address{Roberto Scotto\newline
		Universidad Nacional del Litoral, FIQ.\newline Santiago del Estero 2829,\newline S3000AOM, Santa Fe, Argentina}
	\email{roberto.scotto@gmail.com}

	\thanks{The authors were supported by CAI+D (UNL)}
	
	

	%%% ----------------------------------------------------------------------
	
	\begin{abstract}We give sufficient conditions on the exponent $p: \mathbb R^d\rightarrow [1,\infty)$ for the boundedness of the non-centered Gaussian maximal function on variable Lebesgue spaces $L^{p(\cdot)}(\gamma_d)$, as well as of the new higher order Riesz transforms associated with the Ornstein-Uhlenbeck semigroup, which are the natural extensions of the supplementary first order Gaussian Riesz transforms defined by A. Nowak and K. Stempak in \cite{nowakstempak}. 
	\end{abstract}
	
	%%% ----------------------------------------------------------------------
	\maketitle
	%%% ----------------------------------------------------------------------

\section{Introduction}

 Gaussian harmonic analysis on $\mathbb{R}^d$ is represented by a differential operator called Ornstein-Uhlenbeck which is defined as
\begin{equation*}
\mathcal L:=-\frac{1}{2}\Delta + x\cdot \nabla,
\end{equation*}
where $\Delta =\sum_{i=1}^d \frac{\partial^2}{\partial x_i^2}$ is the Laplacian and $\nabla=\left(\frac{\partial}{\partial x_i}\right)_{i=1}^d$  is the classical gradient.
This operator factors out on each variable into  two derivatives as follows. Indeed, by naming $\delta_i=\frac{1}{\sqrt{2}}\frac{\partial}{\partial x_i}$  and 
$\delta^*_i=-\frac{1}{\sqrt{2}}e^{|x|^2}\frac{\partial }{\partial x_i}\left( e^{-|x|^2}\cdot \right)$, the formal adjoint of $\delta_i$ with respect to the $d$-dimensional non-standard Gaussian measure $d\gamma_d(x)=\frac{e^{-|x|^2}}{\pi^{d/2}}dx$,  we have the differential operator $\mathcal L_i=\delta_i^*\delta_i,$ and
\[\mathcal L=\sum_{i=1}^d \mathcal L_i.\]

Let us remark that $\mathcal L$ is an unbounded non-negative symmetric operator on $L^2(\gamma_d)$. Besides, there is a dense linear subspace of this space where $\mathcal L$ turns out to be a self-adjoint operator (see \cite{Grigoryan}). It has a discrete spectrum $\sigma(\mathcal L)=\{0,1,2, 3,\ldots\}:=\mathbb{N}_0$ which are all eigenvalues for $\mathcal{L}$ and its eigenfunctions are the $d$-dimensional Hermite polynomials (see, for instance, \cite{Davies} for the one-dimensional case, and \cite{Urbina} for higher dimensions).

Following the notation of \cite{nowakstempak} we have $d$ more differential operators which are associated with the $\delta_i^*$-derivatives. For $i=1,\dots,d$ let us define
\begin{equation*}
M_i:=\mathcal L+[\delta_i,\delta_i^*]I,
\end{equation*}
where $I$ is the identity operator and $[\delta_i,\delta_i^*]=\delta_i\delta_i^*-\delta_i^*\delta_i$ is the commutator corresponding to the $i$-th derivatives. In this case, $[\delta_i,\delta_i^*]=1$ for all $i.$ Thus $M_1=M_2=\cdots =M_d=:\bar{\mathcal L}=\mathcal L+I.$ This differential operator has also a discrete spectrum, being $\sigma(\overline{\mathcal L})=\mathbb{N}.$

Associated with these two differential operators there exist two diffusion semigroups, i.e. $T^t=e^{-\mathcal Lt}$ and $\overline{T}^t=e^{-\overline{\mathcal L}t},$ which are defined through the spectral decomposition of $\mathcal L$ and $\overline{\mathcal L}$ on $L^2(\gamma_d)$, respectively.

In order to define the Gaussian Riesz transforms, in this article we are going to consider two different transforms. We need the fractional integrals, say, for $\beta>0,$ 
\begin{align*} I_\beta&=\mathcal L^{-\beta}=\frac{1}{\Gamma(\beta)}\int_0^\infty t^{\beta-1}T^t dt  \\
\overline{I}_\beta &=\overline{\mathcal L}^{-\beta}=\frac{1}{\Gamma(\beta)}\int_0^\infty t^{\beta-1}\overline{T}^t dt.
\end{align*} 
We set $I_0=\bar{I}_0=I.$

Now we define the two $i$-th Gaussian Riesz transforms (see \cite{nowakstempak}). For $i=1,\dots,d$, let
\begin{align*}
R_i f(x)&=\delta_i I_{\frac12}f(x),\\
R_i^*f(x)&=\delta_i^* \overline{I}_{\frac12}f(x).
\end{align*}
Like in classical harmonic analysis, these Gaussian Riesz transforms verify the following equation
\begin{equation*}
\sum_{i=1}^d R_i^* R_i=I	
\end{equation*}
on $L^2(\gamma_d),$ that is, they decompose the identity.

Now, we want to define higher order Gaussian Riesz transforms that retain this property. Let us introduce some notation.
For a multi-index $\alpha=(\alpha_1,\cdots, \alpha_d)\in \mathbb{N}_0^d,$ we define $|\alpha|=\alpha_1+\alpha_2+\cdots+\alpha_d,$  $\delta^\alpha=\delta_1^{\alpha_1}\delta_2^{\alpha_2}\cdots \delta_d^{\alpha_d},$ and similarly $\delta_*^\alpha=(\delta_1^*)^{\alpha_1}(\delta_2^*)^{\alpha_2}\cdots (\delta_d^*)^{\alpha_d}$, where we set $\delta_i^0=(\delta_i^*)^0=I$ for $i=1,\dots,d.$
Thus, we are ready to define the higher order Gaussian Riesz transforms as follows
\begin{align*}
R_\alpha f(x)&=\delta^\alpha I_{\alpha_1/2}I_{\alpha_2/2}\cdots I_{\alpha_d/2}f(x), \\
R^*_\alpha f(x)&=\delta_*^\alpha \overline{I}_{\alpha_1/2}\overline{I}_{\alpha_2/2}\cdots \overline{I}_{\alpha_d/2}f(x).
\end{align*}
Taking into account that $I_\beta I_\epsilon=I_{\beta+\epsilon}$ and $\overline{I}_\beta \overline{I}_\epsilon=\overline{I}_{\beta+\epsilon},$ we can rewrite the Gaussian Riesz transforms as
\begin{align*}
R_\alpha f(x)&=\delta^\alpha I_{|\alpha|/2}f(x), \\
R^*_\alpha f(x)&=\delta_*^\alpha \overline{I}_{|\alpha|/2}f(x).
\end{align*}
Let us remark that $R_{e_i}=R_i$ and $R^*_{e_i}=R_i^*$, where $e_i$ is the $i$-th canonical unit vector of $\mathbb{N}_0^d$. Let us also note that W. Urbina-Romero in \cite{Urbina} (see also \cite{NPU}) has defined alternative higher order Gaussian Riesz transforms $\overline{R}_\beta$ but he does not recover the suitable supplementary first order Gaussian Riesz transforms $R_i^*$ given by A. Nowak and K. Stempak \cite{nowakstempak} when $\beta=e_i$.

We will refer to $R_\alpha$ as the ``old'' Gaussian Riesz transform, and to $R_\alpha^*$ as the ``new'' Gaussian Riesz transform. The reason why we are considering the word ``new'' added to the higher order Gaussian Riesz transforms is because they  were firstly used in \cite{AFS}, in order to distinguish them from the existed first ones which were extensively studied before. The difference between them is in the choice of the derivatives in which the Ornstein-Uhlenbeck differential operator is factored out.

The operator $R_\alpha$ turns out to be bounded on $L^p(\gamma_d),$ for $1<p<\infty,$ with constant independent of dimension (see \cite{meyer}, \cite{gutierrezsegoviatorrea}, \cite{forzaniscottourbina}). For the first order Gaussian Riesz transforms $R_i^*$, $i=1,\dots,d$, $L^p(\gamma_d)$-dimension-free estimates were obtained in \cite{FSS} and \cite{Wrobel}, for $1<p<\infty$.  By means of Meyer's multiplier theorem, the ``new'' higher order Riesz transforms are also bounded on $L^p(\gamma_d)$, as can be proved similarly to \cite[Corollary 9.14]{Urbina}, with constant independent of dimension. According to \cite{bogachev18}, the Euclidean space $\mathbb{R}^d$ can be extended to an infinite-dimensional real Hausdorff locally convex space $X$, where one can introduce a standard Gaussian measure $\gamma$. In this context we can define the analogous diffusion semigroup $T^t$ whose infinitesimal generator is the Ornstein-Uhlenbeck operator $\mathcal{L}=\text{div}_\gamma D_H$, being $D_H$ the gradient on a Cameron-Martin space $H$. Taking into account these derivatives and the potentials associated with the corresponding Sobolev spaces, one can define singular integrals on this context and the boundedness of them on $L^p(\gamma)$ can be obtained from their boundedness on $L^p(\gamma_d)$, with constants independent of dimension.

In \cite{DS}, we have proved that each $R_\alpha$ is bounded on variable Lebesgue spaces with respect to the non-standard   Gaussian measure. Here, we cannot obtain a constant independent of dimension. It is an open problem to find a technique similar to the Littlewood-Paley one that gives a boundedness independent of dimension.  Inspired by \cite{DS}, the main aim of the present article is to show the following boundedness property.
\begin{thm}\label{thm: main}The new Gaussian Riesz transforms $R_\alpha^*$ are bounded on $L^{p(\cdot)}(\gamma_d)$ provided that $p^->1$ and $p\in LH_0(\mathbb{R}^d)\cap\mathcal{P}_{\gamma_d}^\infty(\mathbb{R}^d)$.
\end{thm}
For the definitions and notations involved in the theorem above, see \S \ref{sec: prelim}.

In order to get a proof of Theorem \ref{thm: main}, in \S \ref{sec: new-variable} we will introduce a more general operator which contains these Riesz transforms and show its boundedness, following closely the ideas of our previous article \cite{DS}. Indeed, Theorem \ref{thm: main} will be a particular case of Theorem \ref{thm: local+global}.

On the other hand, in \S \ref{sec: maximal-variable}, we consider the variable $L^{p(\cdot)}$ boundedness of the non-centered Gaussian Hardy-Littlewood maximal function $\mathcal{M}_{\gamma_d}.$ Moreover, we prove a more general theorem dealing with such a boundedness for a maximal operator associated to a measure $\mu$ defined on a metric space, finding a geometric condition on the measure $\mu$ over the balls similar to the L. Diening's condition \eqref{eq: Diening-geom} for the case of Lebesgue measure. We prove that the very same properties on the exponents required in Theorem \ref{thm: main} are also sufficient for its boundedness on $L^{p(\cdot)}(\gamma_d)$. An equivalent condition to $\mathcal{P}_{\gamma_d}^\infty(\mathbb{R}^d)$ is also established.

\section{Preliminaries}\label{sec: prelim}

We now give some definitions about variable Lebesgue spaces on a measure space.

Given a $\sigma$-finite measure $\mu$ over $\mathbb R^d$, we shall denote with $\mathcal{P}(\mathbb R^d,\mu)$ the set of \textit{exponents}, that is, the set of $\mu$-measurable and bounded functions $p:\mathbb{R}^d\rightarrow [1,\infty)$. When $\mu$ is the Lebesgue measure, we simply write $\mathcal P(\mathbb{R}^d)$. For a $\mu$-measurable set $E$, we will write
\[p^-_E=\underset{x\in E}{\operatorname*{ess\ inf}}\ p(x), \quad p^+_E=\underset{x\in E}{\operatorname*{ess\ sup}}\ p(x),\]
and, for the whole space, we denote $p^-=p^-_{\mathbb{R}^d}$ and $p^+=p^+_{\mathbb{R}^d}$. 

Given $p\in \mathcal{P}(\mathbb{R}^d,\mu)$, we say that a $\mu$-measurable function $f$ belongs to $L^{p(\cdot)}(\mathbb{R}^d,\mu)$ if, for some $\lambda >0$,
\[\int_{\mathbb{R}^d} \left(\frac{|f(x)|}{\lambda}\right)^{p(x)}d\mu(x) <\infty.\]
The natural norm for these spaces is the Luxemburg norm, defined by
\[||f||_{p(\cdot),\mu}=\inf\bigg\{\lambda>0: \int_{\mathbb{R}^d} \left(\frac{|f(x)|}{\lambda}\right)^{p(x)}d\mu(x)\le 1 \bigg\},\]
which recovers the classical norm $||f||_{p,\mu}=\left(\int_{\mathbb{R}^d} |f(x)|^p d\mu(x)\right)^{1/p}$ when $p(x)\equiv p$, $1\le p<\infty$. It is also well-known that $\left(L^{p(\cdot)}(\mathbb{R}^d,\mu),||\cdot||_{p(\cdot),\mu}\right)$ is a Banach function space (\cite[Theorem 3.2.7]{DHHR}). When $\mu$ is the classical Lebesgue measure, we simply write $L^{p(\cdot)}(\mathbb{R}^d)$ and the norm as $||\cdot||_{p(\cdot)}$. 

Associated with each exponent $p\in \mathcal P(\mathbb{R}^d, \mu)$, we have another exponent $p'\in \mathcal P(\mathbb{R}^d, \mu)$, which is the generalization to the variable context of H\"older's conjugate exponent given by
\[\frac{1}{p(x)}+\frac{1}{p'(x)}=1, \quad\forall \, x\in \mathbb{R}^d.\] 
As expected, a generalization of H\"older's inequality  holds for variable exponents (\cite[Lemma 3.2.20]{DHHR}). Given a measure $\mu$ as above, for every pair of functions $f\in L^{p(\cdot)}(\mathbb{R}^d,\mu)$ and $g\in L^{p'(\cdot)}(\mathbb{R}^d,\mu)$,
\begin{equation*}\int_{\mathbb{R}^d} |f(x)g(x)|d\mu(x)\le 2\|f\|_{p(\cdot),\mu}\|g\|_{p'(\cdot),\mu}.
\end{equation*}

Another important property is the norm conjugate formula: for any $\mu$-measurable function $f$, the following inequalities
\begin{equation*}
\frac{1}{2}\|f\|_{p(\cdot),\mu}\le \sup\limits_{\|g\|_{p'(\cdot),\mu}\le 1} \int_{\mathbb{R}^d}|f(x)g(x)|d\mu(x)\le 2 \|f\|_{p(\cdot),\mu}.
\end{equation*}
hold (\cite[Corollary~3.2.14]{DHHR}). For more information about $L^{p(\cdot)}$ spaces, see, for instance, \cite{CUF} or \cite{DHHR}.

The measure we shall be dealing with is the non-standard Gaussian measure $\gamma_d$, which is a finite, non-doubling and upper Ahlfors $d$-regular measure on $\mathbb R^d$. From now on, $\mu=\gamma_d$.

The exponents we will consider are not arbitrary, but we may allow them to have some continuity properties. The following conditions on the exponent arise related with the boundedness of the Hardy-Littlewood maximal function $M_{\text{H-L}}$ on $L^{p(\cdot)}(\mathbb{R}^d)$ (see, for example, \cite{CUFN} or \cite{D1}).

\begin{defn}
	Let $p\in \mathcal{P}(\mathbb R^d)$. 
	\begin{enumerate}
		\item  We will say that $p\in LH_0(\mathbb{R}^d)$ if there exists $C_{\log}(p)>0$ such that, for any pair $x,y\in \mathbb{R}^d$ with $0<|x-y|<1/2$,
		\begin{equation}\label{def: LH0}
			|p(x)-p(y)|\le \frac{C_{\log}(p)}{-\log(|x-y|)}.
		\end{equation}
		
		\item We will say that $p\in LH_\infty(\mathbb{R}^d)$ if there exist constants $C_\infty>0$ and $p_\infty\ge 1$ for which
		\begin{equation}\label{def: LHinfty}
			|p(x)-p_\infty|\le \frac{C_\infty}{\log(e+|x|)}, \quad \forall \, x\in \mathbb{R}^d.
		\end{equation}
	\end{enumerate}
We will say $p\in LH(\mathbb R^d)$ when $p\in LH_0(\mathbb R^d)\cap LH_\infty(\mathbb R^d)$.
\end{defn}

Conditions \eqref{def: LH0} and \eqref{def: LHinfty} are usually called the \textit{local log-H\"{o}lder condition} and the \textit{decay log-H\"{o}lder condition}, respectively. When $p\in LH(\mathbb R^d)$, we simply say that it is \textit{log-H\"older continuous}. It is well-known that whenever $1<p^-\le p^+<\infty$, $p\in LH(\mathbb{R}^d)$ is a sufficient condition for the continuity of the Hardy-Littlewood maximal operator ${M}_{\text{H-L}}$ on variable Lebesgue spaces (see, for instance, \cite{CUFN}). However, it is not a necessary condition although it was proved in \cite[Examples 4.1 and 4.43]{CUF} that both $LH_0(\mathbb R^d)$ and $LH_\infty(\mathbb R^d)$ are the sharpest possible pointwise conditions on $p$. The authors in \cite{DHHR} gave a necessary and sufficient condition for the $L^{p(\cdot)}$-boundedness of ${M}_{\text{H-L}}$, but it is not easy to work with from the practical point of view. In this article, we do not expect to characterize the exponents, but to give sufficient easy-to-check conditions for them in order to obtain the boundedness properties for the operators in study.

Regarding the local log-H\"older condition \eqref{def: LH0}, L. Diening gave a geometric characterization of it (see \cite{D1}) in order to prove the boundedness of $M_{\text{H-L}}$ on bounded subsets of $\mathbb R^d$ or in the whole Euclidean space assuming $p$ is constant outside of a fixed ball.

\begin{lem}[\cite{D1}] Given $p\in \mathcal P(\mathbb R^d)$, $p\in LH_0(\mathbb R^d)$ if and only if there exists a positive constant $C$ such that
	\begin{equation}\label{eq: Diening-geom}
	|B|^{p^+_B-p^-_B}\ge C,
	\end{equation}
for every ball $B$.
\end{lem}

An analogous property can also be obtained when dealing with the boundedness of $\mathcal{M}_\mu$, the non-centered maximal function associated with the measure $\mu$. We will consider it on \S \ref{sec: maximal-variable}.

Whilst Diening's geometric condition can be applied to control the behavior of $f$ when it is large, condition \eqref{def: LHinfty} happens to be useful when $f$ is small. This is evidenced in the following lemma, which establishes that we can change a variable exponent $p$ for its limit $p_\infty$, and viceversa, adding an integrable error (for a proof, see for instance \cite[Lemma~3.26]{CUF}). Previous results of this kind were given in \cite{CUFN, CUFNerrata}, closely related with the boundedness of the Hardy-Littlewood maximal operator in the Euclidean setting.

\begin{lem}[\cite{CUF}]\label{lem: changep}Let $p\in LH_\infty(\mathbb{R}^d)$ with $1<p^-\le p^+<\infty$. Then, there exists a constant C, depending on $d$ and $C_\infty$, such that for any set $E$ and any function $G$ with $0\le G(y)\le 1$ for $y\in E$,
	\begin{align*}
	\int_E G(y)^{p(y)} dy &\le C \int_E G(y)^{p_\infty} dy+\int_E (e+|y|)^{-dp^-} dy,\\
	\int_E G(y)^{p_\infty} dy &\le C \int_E G(y)^{p(y)} dy+\int_E (e+|y|)^{-dp^-} dy.
	\end{align*}
\end{lem}

We will now recall the class of exponents introduced in our previous article \cite{DS}, which is related with the boundedness of the ``old'' Riesz transforms.
\begin{defn}
	 Given $p\in \mathcal P(\mathbb{R}^d, \gamma_d)$, we will say that $p\in \mathcal P_{\gamma_d}^\infty(\mathbb{R}^d)$ if there exist constants $C_{\gamma_d}>0$ and $p_\infty\ge 1$ such that
	\begin{equation*}
		|p(x)-p_\infty|\le \frac{C_{\gamma_d}}{|x|^2}, \quad \forall \,  x\in \mathbb{R}^d\setminus \{(0,\dots,0)\}.
	\end{equation*}
\end{defn}

As observed in \cite[Remark 2.4]{DS}, if $p\in \mathcal{P}_{\gamma_d}^\infty(\mathbb{R}^d)$, then $p\in LH_\infty(\mathbb{R}^d)$, and, if $p^->1$, also $p'\in \mathcal{P}_{\gamma_d}^\infty(\mathbb{R}^d)$ with $(p')_\infty=(p_\infty)':=p'_\infty<\infty$. Since, in this case,  $p_\infty=\underset{|x|\rightarrow \infty}{\lim}p(x)$, we have $p_\infty>1$ whenever $p^->1$.

 From now on, we will use this notation: given two functions $f$ and $g$, by $\lesssim$ and $\gtrsim$ we will mean that there exists a positive constant $c$ such that $f\le c g$ and $cf\ge g$, respectively. When both inequalities hold, i.e., $f\lesssim g \lesssim f$, we will write it as $f\thickapprox g$.
 
 \bigskip
 
 As it is usual in the Gaussian context, we consider the ``local'' and ``global'' parts of several operators, in order to analyze their properties. For this partition, we may recall the definition of the hyperbolic ball
 \[B(x):=\left\{y\in \mathbb R^d: |y-x|\le d(1\wedge 1/|x|)\right\},\quad x\in \mathbb R^d,\]
 where $\alpha\wedge \beta=\min\{\alpha,\beta\}, \alpha,\beta\in \mathbb{R}$. Given a sublinear operator $S$, we say that $S\left(f\chi_{B(\cdot)}\right)$ is the \textit{local part} and $S\left(f\chi_{B^c(\cdot)}\right)$, being $B^c(x):=\mathbb{R}^d\setminus B(x)$, is the \textit{global part}.

\section{The ``new'' higher order Gaussian Riesz Transforms on variable Lebesgue spaces}\label{sec: new-variable}

We have that the ``old'' higher order Gaussian Riesz transforms $R_\alpha f$ can be written as an integral operator with a kernel $K_\alpha$
\begin{equation*}
R_\alpha f(x)=\text{p.v.} \int_{\mathbb{R}^d}  K_\alpha(x,y)f(y)dy,
\end{equation*}
with \[ K_\alpha(x,y)=C_\alpha\int_0^1 r^{|\alpha|-1}\left(\frac{-\log r}{1-r^2}\right)^{\frac{|\alpha|-2}{2}} H_\alpha \left(\frac{y-rx}{\sqrt{1-r^2}}\right) \frac{e^{-\frac{|y-rx|^2}{1-r^2}}}{(1-r^2)^{\frac{d}{2}+1}} dr.\]
On the other hand, the ``new'' higher order Gaussian Riesz transforms are given by
\[R_\alpha^* f(x)=\text{p.v.} \int_{\mathbb{R}^d}  \overline{ K}_\alpha(x,y)f(y)dy,\]
where 
\[ \overline{ K}_\alpha(x,y)=C_\alpha\int_0^1 \left(\frac{-\log r}{1-r^2}\right)^{\frac{|\alpha|-2}{2}} H_\alpha \left(\frac{x-ry}{\sqrt{1-r^2}}\right) \frac{e^{-\frac{|x-ry|^2}{1-r^2}}}{(1-r^2)^{\frac{d}{2}+1}} dr\ e^{|x|^2-|y|^2}\]
and $\alpha$ is a multi-index in $\mathbb N_0^d\setminus\{(0,\dots,0)\}$.

As we said in the introduction, in the spirit of \cite{DS} and \cite{sonsoles}, we will introduce a larger class of singular integrals, containing $R_\alpha^*$, and prove their boundedness on $L^{p(\cdot)}(\gamma_d)$.

Let $F\in C^1(\mathbb{R}^d)$ be a function which is orthogonal with respect to the Gaussian measure, i.e. \[\int_{\mathbb{R}^d} F(x) d\gamma_d(x)=0,\] and for every $\epsilon>0,$ there exists some constant $C_\epsilon >0,$ such that $\forall x\in \mathbb{R}^d$
\begin{enumerate}[label=(\roman*)]
	\item $|F(x)|\le C_\epsilon \ e^{\epsilon |x|^2},$
	\item $|\nabla F(x)|\le C_\epsilon \ e^{\epsilon |x|^2}$.
\end{enumerate}

We define the singular integral operator 
\begin{equation*}
\overline{R}_F f(x)=\text{p.v.}\int_{\mathbb{R}^d}\overline{ K}_F(x,y) f(y) dy,
\end{equation*}
with kernel \[\overline{ K}_F(x,y)=\int_0^1 \left(\frac{-\log r}{1-r^2}\right)^{\frac{m-2}{2}} F \left(\frac{x-ry}{\sqrt{1-r^2}}\right) \frac{e^{-\frac{|x-ry|^2}{1-r^2}}}{(1-r^2)^{\frac{d}{2}+1}} dr\ e^{|x|^2-|y|^2}.\]
If we make the change of variables $t=1-r^2$ in the integral defining the above kernel we obtain
\[\overline{ K}_F(x,y)=\int_0^1 \psi_m (t)\ F\left(\frac{x-\sqrt{1-t}y}{\sqrt{t}}\right)\frac{e^{-\frac{|x-\sqrt{1-t}y|^2}{t}}}{t^{\frac{d}{2}+1}}\frac{dt}{\sqrt{1-t}}\ e^{|x|^2-|y|^2},\]
with $\psi_m(t)=\left(\frac{\log \frac{1}{\sqrt{1-t}}}{t}\right)^{\frac{m-2}{2}}=\left(\frac{-\log (1-t)}{t}\right)^{\frac{m-2}{2}}2^{-\frac{m-2}{2}}.$

If we set $F(x)=C_\alpha H_\alpha (x),$ then $\overline{R}_F=R_\alpha^*,$ with $m=|\alpha|=\alpha_1+\cdots+\alpha_d.$

The singular integral will be splitting into the local and global parts as follows
\begin{align*}
\overline{R}_F f(x)&=\overline{R}_F (f\chi_{B(\cdot)})(x)+\overline{R}_F (f\chi_{B^c(\cdot)})(x)\\ 
& =:\boldsymbol{L}f(x) +\boldsymbol{G} f(x).
\end{align*}

\subsection{The local part}

In this section, we shall prove that
\begin{lem}For every $x\in \mathbb R^d$, 
	\begin{equation}\label{eq: localpart}
	\boldsymbol{L}f(x)\lesssim \sum_{B\in \mathcal{F}}\left(|\overline{T}_F(f\chi_{\hat{B}})(x)|+ M_{\text{H-L}}(f\chi_{\hat{B}})(x)\right)\chi_B(x),
	\end{equation}
being $\overline{T}_F$ a singular integral operator, $M_{\text{H-L}}$ the non-centered Hardy-Littlewood maximal function, and $\mathcal F=\{B\}$ and $\hat{\mathcal{F}}=\{\hat{B}\}$ are the families of balls given by \cite[Lemma 3.1]{DS}.
\end{lem}

We are going to look at the kernel written as
\begin{align*}\overline{ K}_F(x,y)&=\int_0^1 \frac{\psi_m (t)}{\sqrt{1-t}}\ F\left(\frac{x-\sqrt{1-t}y}{\sqrt{t}}\right)\frac{e^{-\frac{|y-\sqrt{1-t}x|^2}{t}}}{t^{\frac{d}{2}+1}}\ dt \\ &=
\int_0^1 \phi_m (t)  F\left(\frac{x-\sqrt{1-t}y}{\sqrt{t}}\right)\frac{e^{-\frac{|y-\sqrt{1-t}x|^2}{t}}}{t^{\frac{d}{2}+1}}\ dt
\end{align*}
where, as before, $\psi_m(t)=\left(\frac{-\log (1-t)}{t}\right)^{\frac{m-2}{2}}2^{-\frac{m-2}{2}},$ and $\phi_m(t)=\frac{\psi_m(t)}{\sqrt{1-t}}.$ Let us remark that $\lim_{t\to0^+}\phi_m(t)=2^{-\frac{m-2}{2}},$ so we can define $\phi_m(0)=2^{-\frac{m-2}{2}}.$ Besides \[|\phi_m(t)-\phi_m(0)|\le C \frac{t}{1-t}.\]
Then
\begin{align*}\overline{K}_F(x,y)&=\left(\int_0^{\frac12}+\int_{\frac12}^1\right) \phi_m (t)  F\left(\frac{x-\sqrt{1-t}y}{\sqrt{t}}\right)\frac{e^{-\frac{|y-\sqrt{1-t}x|^2}{t}}}{t^{\frac{d}{2}+1}}\ dt\\
&=:\overline{K}_F^1(x,y)+\overline{K}_F^2(x,y). \end{align*}
Let 
\begin{equation}\label{eq: u}
u(t):=\frac{|y-\sqrt{1-t}x|^2}{t},
\end{equation}
and
\begin{equation}\label{eq: ubar}
\overline{u}(t):=\frac{|x-\sqrt{1-t}y|^2}{t}.
\end{equation}
Then $\overline{u}(t)=u(t)+|x|^2-|y|^2.$  On $B(x),$ $u(t)\ge \frac{|x-y|^2}{t}-2d$ and $||x|^2-|y|^2|\le 4 d.$ Thus 
\begin{align}\label{eq: KF2}
|\overline{K}_F^2(x,y)|&\lesssim \int_{\frac12}^1\frac{(-\log (1-t))^{\frac{m-2}{2}}}{\sqrt{1-t}}e^{\epsilon \overline{u}(t) }\frac{e^{-u(t)}}{t^{\frac d2}} dt\nonumber\\
&=\int_{\frac12}^1\frac{(-\log (1-t))^{\frac{m-2}{2}}}{\sqrt{1-t}}e^{\epsilon (|x|^2-|y|^2) }\frac{e^{-(1-\epsilon)u(t)}}{t^{\frac d2}} dt\nonumber \\ &\lesssim \int_{0}^1\frac{(-\log (1-t))^{\frac{m-2}{2}}}{\sqrt{1-t}}\frac{e^{-(1-\epsilon)\frac{|x-y|^2}{t}}}{t^{\frac d2}} dt \nonumber\\
 &\lesssim \left(\int_0^\infty s^{\frac{m-2}{2}}e^{-\frac{s}{2}}ds\right) \sup_{t>0}\omega_{\sqrt{t}}(x-y),\end{align}
where we have applied the change of variables $s=-\log(1-t)$, and  considered $\omega_{t}(z):=t^{-d}e^{-(1-\epsilon)\frac{|z|^2}{t^2}}$.

 Now \begin{align*}
\overline{K}_F^1(x,y)&= \phi_m(0) \int_0^1 F\left(\frac{x-\sqrt{1-t}y}{\sqrt{t}}\right)\frac{e^{-u(t)}}{t^{\frac d2+1}}dt\\&\quad -\phi_m(0) \int_{\frac12}^1 F\left(\frac{x-\sqrt{1-t}y}{\sqrt{t}}\right)\frac{e^{-u(t)}}{t^{\frac d2+1}}dt\\ &\quad +
\int_0^{\frac12}(\phi_m(t)-\phi_m(0))F\left(\frac{x-\sqrt{1-t}y}{\sqrt{t}}\right)\frac{e^{-u(t)}}{t^{\frac d2+1}}dt\\ & =: \widetilde{K}(x,y)+\overline{K}_{F,1}^1(x,y)+\overline{K}^1_{F,2}(x,y).
\end{align*}
Let us observe that, for $j=1,2$, we can proceed as before to obtain
\[|\overline{K}^1_{F,j}(x,y)|\lesssim e^{\epsilon\left(|x|^2-|y|^2\right)}\int_0^1 \frac{e^{-(1-\epsilon)\frac{|x-y|^2}{t}}}{t^{\frac d2}}dt\lesssim \sup_{t>0}\omega_{\sqrt{t}}(x-y).\]

Following \cite{sonsoles}, if we define, for $x\ne 0,$ \[\overline{\mathcal{K}}_F(x)=\int_0^\infty F\left(\frac{x}{\sqrt{t}}\right)\frac{e^{-\frac{|x|^2}{t}}}{t^{\frac d2+1}}dt=\frac{\Omega(x')}{|x|^d},\] with $\Omega(x')=2\int_0^\infty F(sx')s^{d-1}e^{-s^2}ds$ and $x'=\frac{x}{|x|},$ we have
\[\int_{S^{d-1 }}\Omega(x')d\sigma(x')=2\pi^{\frac d2}\int_{\mathbb{R}^d}F(x)d\gamma_d(x)=0,\] i.e., $\overline{\mathcal{K}}_F(x)$ is a homogeneous kernel of degree $-d$, and therefore 
\begin{equation}\label{eq: T-CZ}
	\overline{T}_F(f)(x)=\text{p.v.} \ \overline{\mathcal{K}}_F\ast f(x)
\end{equation}
is a singular integral operator with homogeneous kernel, an example of a singular integral of Calder\'on-Zygmund type.

Now we write 
\[\overline{K}_F^1(x,y)=\overline{\mathcal{K}}_F(x-y)+ \overline{K}_{F,0}^1(x,y) +\overline{K}_{F,1}^1(x,y)+\overline{K}^1_{F,2}(x,y),\]
with $\overline{K}_{F,0}^1(x,y)=\widetilde{K}(x,y)-\overline{\mathcal{K}}_F(x-y).$ 

Let us recall that for every $x\in B\in \mathcal{F},$ $B(x)\subset \hat{B}$. Hence, looking at \cite[p. 506]{sonsoles} (see also \cite[(3.4) and (3.8)]{DS}) and taking into account that on $B(x),$ $|x|\thickapprox |y|$, we have
\begin{align*}
\int_{B(x)}\overline{K}_{F,0}^1(x,y) |f(y)|dy&\lesssim \left(1+|x|^{\frac12}\right) \int_{\mathbb{R}^d}\frac{1}{|x-y|^{d-1/2}}|f(y)|\chi_{\hat{B}}(y)dy \\ &\lesssim M_{\text{H-L}}(f \chi_{\hat{B}})(x),
\end{align*}
for $j=1,2$
\begin{equation*}
\int_{B(x)}\overline{K}_{F,j}^1(x,y) |f(y)|dy\lesssim \int_{\mathbb{R}^d} \sup_{t>0} \omega_{\sqrt{t}}(x-y)|f(y)|\chi_{\hat{B}}(y)dy\lesssim M_{\text{H-L}}(f \chi_{\hat{B}})(x),
\end{equation*}
and also, from \eqref{eq: KF2},
\begin{equation*}
\int_{B(x)}\overline{K}_{F}^2(x,y) |f(y)|dy\lesssim \int_{\mathbb{R}^d} \sup_{t>0} \omega_{\sqrt{t}}(x-y)|f(y)|\chi_{\hat{B}}(y)dy\lesssim M_{\text{H-L}}(f \chi_{\hat{B}})(x).
\end{equation*}
Finally, for $x\in B$,
\[
\text{p.v.} \int_{B(x)}\overline{\mathcal{K}}_F(x-y)f(y)dy=\overline{T}_F(f\chi_{\hat{B}})(x)- \int_{\hat{B}\setminus B(x)}\overline{\mathcal{K}}_F(x-y)f(y)dy,\]
and $\left|\int_{\hat{B}\setminus B(x)}\overline{\mathcal{K}}_F(x-y)f(y)dy\right|\lesssim M_{\text{H-L}}(f\chi_{\hat{B}})(x).$ 

This ends the estimates for the local part of the operator and \eqref{eq: localpart} yields.

\subsection{The global part}

The aim of this section is to show that
\begin{lem}For every $x\in \mathbb R^d$, and $0<\epsilon<\frac{1}{2p'_\infty}\wedge \frac1d$,
	\begin{equation}\label{eq: globalpart}
	\boldsymbol{G}f(x)\lesssim  e^{\epsilon |x|^2}\left(\int_{\mathbb{R}^d}|f(y)|^{p^{-}}\gamma_d(dy)\right)^{\frac{1}{p^{-}}} +e^{\frac{|x|^2}{p(x)}}\int_{B(x)^c}P(x,y)|f(y)|e^{-\frac{|y|^2}{p(y)}}dy 
	\end{equation}
	being \[P(x,y)=|x+y|^d e^{-\alpha_\infty |x-y||x+y|}\] where $\alpha_\infty =\frac{1-\epsilon}{2}-\left|\frac{1}{p_\infty}-\frac{1-3\epsilon}{2}\right|$. 
\end{lem}

Let us notice that the global part of these new Gaussian Riesz transforms is strictly larger than the global part of the old ones, but still we get the right estimates on this kernel such that the boundedness of this part also holds. 

In order to study the global part of the new higher order Gaussian Riesz transforms, we will follow the ideas of \cite{sonsoles}. To that end, we might recall some notation and results from that article (see also \cite{MPS2} or \cite{MPS1}).

For $x,y\in \mathbb R^d$, we set
\[a=a(x,y)=: |x|^2+|y|^2\quad \textrm{and}\quad b=b(x,y)=:2\langle x,y\rangle.\]

On the complement of $B(x)$, we know $a>d/2$ and $\sqrt{a^2-b^2}=|x+y||x-y|>d$ whenever $b>0$.

Recall the definitions of $u$ and $\overline{u}$ given in \eqref{eq: u} and \eqref{eq: ubar}. Hence, we can write $u(t)=\frac{a}{t}-\frac{\sqrt{1-t}}{t}b-|y|^2$. Both $u$ and $\overline{u}$ 
have a minimum and it is attained at $t_0$, given by
\begin{equation}\label{eq: t0}
t_0=\left\{\begin{tabular}{ll} $2\frac{\sqrt{a^2-b^2}}{a+\sqrt{a^2-b^2}}$& if $b>0$\\ $1$& if $b\le 0$.\end{tabular}\right.
\end{equation}
The minimum value is 
\begin{equation}\label{eq: u0}
u_0:=u(t_0)=\left\{\begin{tabular}{ll} $\frac{|y|^2-|x|^2+|x+y||x-y|}{2}$& if $b>0$\\ $|y|^2$& if $b\le 0$.\end{tabular}\right.
\end{equation}
Then, 
\[\frac{e^{-u(t)}}{t^{d/2}}\le C \frac{e^{-u_0}}{t_0^{d/2}},\] and 
\[\frac{e^{-\bar{u}(t)}}{t^{d/2}}\le C \frac{e^{-(u_0+|x|^2-|y|^2)}}{t_0^{d/2}}.\] 
Moreover, the following result holds.

\begin{lem}\label{lem: boundsexp}Let us consider the kernel $\overline{ K}_F(x,y)$ in the global part, that is, for $y\in B^c(x)$. We have the following inequalities
	\begin{enumerate}[label=(\roman*)]
		\item \label{it: b<=0}If $b\le 0$, for each $0<\epsilon <1$, there exists $C_\epsilon>0$ such that
		\[|\overline{ K}_F(x,y)|\le C_\epsilon e^{\epsilon|x|^2-|y|^2};\]
		\item \label{it: b>0}If $b> 0$, for each $0<\epsilon <1/d$ there exists $C_\epsilon>0$ such that
		\[|\overline{ K}_F(x,y)|\le C_\epsilon \frac{e^{-(1-\epsilon)u_0}}{t_0^{d/2}}e^{\epsilon\left(|x|^2-|y|^2\right)},\]
		where $t_0$ and $u_0$ are as in \eqref{eq: t0} and \eqref{eq: u0}, respectively.
	\end{enumerate}
\end{lem}

\begin{proof}If $b\le 0,$  then  \[\frac{a}{t}-|x|^2\le u(t)= \frac{a}{t}-\frac{\sqrt{1-t}}{t}b-|x|^2\le \frac{2a}{t}.\]
	\begin{align*}
	\overline{ K}_F(x,y)&=\left (\int_0^{\frac12}+\int_{\frac12}^1\right )\psi_m (t)\ F\left(\frac{x-\sqrt{1-t}y}{\sqrt{t}}\right)\frac{e^{-\overline{u}(t)}}{t^{\frac{d}{2}+1}}\frac{dt}{\sqrt{1-t}}\ e^{|x|^2-|y|^2}\\ &=I+II.
	\end{align*}
If $0<m\le 2,$ $\psi_m$ is bounded on $[0,1)$ and 
\begin{equation}\label{eq: m<2,b<0}
|\overline{ K}_F(x,y)|\lesssim \int_0^1\frac{e^{-(1-\epsilon)u(t)}}{t^{\frac{d}{2}+1}}\frac{dt}{\sqrt{1-t}}\ \  e^{\epsilon\left(|x|^2-|y|^2\right)}.	
\end{equation}
And from \cite[p. 500]{sonsoles}, $|\overline{ K}_F(x,y)|$ is bounded by $ e^{\epsilon |x|^2-|y|^2}.$

On the other hand, if $m>2$, taking into account that $\psi_m(t)$ is bounded on $\left[0,\frac12\right]$, we have
\begin{equation*}
|I|\lesssim \int_0^{\frac12}\frac{e^{-(1-\epsilon)\overline{u}(t)}}{t^{\frac{d}{2}+1}}\frac{dt}{\sqrt{1-t}}\ e^{|x|^2-|y|^2}\lesssim
\int_0^1\frac{e^{-(1-\epsilon)u(t)}}{t^{\frac{d}{2}+1}}\frac{dt}{\sqrt{1-t}}\ e^{\epsilon\left(|x|^2-|y|^2\right)},
\end{equation*}
and
\begin{align*}
|II|&\lesssim \int_{\frac12}^1(-\log(1-t))^{\frac{m-2}{2}}e^{-(1-\epsilon)\overline{u}(t)}\frac{dt}{\sqrt{1-t}}\ e^{|x|^2-|y|^2} \\ &=      \int_{\frac12}^1(-\log(1-t))^{\frac{m-2}{2}}e^{-(1-\epsilon)u(t)}\frac{dt}{\sqrt{1-t}} \ \  e^{\epsilon\left(|x|^2-|y|^2\right)}.
\end{align*}
	As before, $|I|\lesssim e^{\epsilon |x|^2-|y|^2}.$ On the other hand, from \cite{sonsoles} again with the change of variables $s=\frac{a}{t}-a$ we get that
\[|II|\lesssim e^{-(1-\epsilon)|y|^2}\frac{1}{\sqrt{a}} \int_0^a\left(\log \left(1+\frac{a}{s}\right)\right)^{\frac{m-2}{2}}e^{-(1-\epsilon)s}\frac{ds}{\sqrt{s}}\ \ e^{\epsilon\left(|x|^2-|y|^2\right)}\]
which, in turn, by the change of variables $w=\log \left(1+\frac{a}{s}\right),$ can be bounded as
\begin{align*}
|II|&\lesssim \frac{1}{\sqrt{a}}\int_{\log 2}^\infty w^{\frac{m-2}{2}}\frac{a}{(1-e^{-w})^2}\frac{e^{-\frac{w}{2}}(1-e^{-w})^{\frac12}}{\sqrt{a}}dw\ \  e^{\epsilon|x|^2-|y|^2} \\ &\lesssim \int_{0}^\infty w^{\frac{m-2}{2}}e^{-\frac{w}{2}}dw\ \   e^{\epsilon|x|^2-|y|^2}\le C_\epsilon e^{\epsilon|x|^2-|y|^2}. 
\end{align*}

Now, we assume $b>0.$ If $0<m\le 2$, we repeat the estimate of \eqref{eq: m<2,b<0}. For $m>2,$
\begin{align*}
|\overline{ K}_F(x,y)|& \lesssim \int_0^1 \psi_m (t)\ \frac{e^{-(1-\epsilon)\overline{u}(t)}}{t^{\frac{d}{2}+1}}\frac{dt}{\sqrt{1-t}}\ e^{|x|^2-|y|^2} \\& =\left(\int_0^{\frac12}+\int_{\frac12}^1\right) \psi_m (t)\ \frac{e^{-(1-\epsilon)u(t)}}{t^{\frac{d}{2}+1}}\frac{dt}{\sqrt{1-t}}\ e^{\epsilon\left(|x|^2-|y|^2\right)} \\&
=I+II.
\end{align*}
To estimate $I$ we use that $\psi_m$ is bounded and the estimates in \cite[p. 500]{sonsoles} to get
\[I\lesssim \frac{e^{-(1-\epsilon)u_0}}{t_0^{d/2}} e^{\epsilon\left(|x|^2-|y|^2\right)}.\]
For the other term, we have
\begin{align*}
II&\lesssim  \int_{\frac12}^1(-\log(1-t))^{\frac{m-2}{2}}e^{-(1-\epsilon)u(t)}\frac{dt}{\sqrt{1-t}} \ \  e^{\epsilon\left(|x|^2-|y|^2\right)}\\ &\lesssim  \int_{\frac12}^1(-\log(1-t))^{\frac{m-2}{2}}\frac{dt}{\sqrt{1-t}}\ e^{-(1-\epsilon)u_0} \ \  e^{\epsilon\left(|x|^2-|y|^2\right)}\\ &\lesssim \int_{0}^\infty w^{\frac{m-2}{2}}e^{-\frac{w}{2}}dw\ \ \frac{e^{-(1-\epsilon)u_0}}{t_0^{d/2}} \ \  e^{\epsilon\left(|x|^2-|y|^2\right)}\\
&\le C_\epsilon \frac{e^{-(1-\epsilon)u_0}}{t_0^{d/2}} e^{\epsilon\left(|x|^2-|y|^2\right)}.\qedhere
\end{align*}
\end{proof}

Now, we are in position to prove \eqref{eq: globalpart}. From Lemma \ref{lem: boundsexp}\ref{it: b<=0}, \begin{align*}\int_{B^c(x)\cap \{b\le 0\}}|\overline{ K}_F(x,y)| |f(y)|dy&\lesssim e^{\epsilon |x|^2}\int_{\mathbb{R}^d}|f(y)|d\gamma_d(y) \\ 
&\lesssim e^{\epsilon |x|^2}\left(\int_{\mathbb{R}^d}|f(y)|^{p^-}d\gamma_d(y)\right)^{1/p^-}.
\end{align*}

On the other hand, from Lemma \ref{lem: boundsexp}\ref{it: b>0}
\begin{align}\label{eq: KF,b>0}
\int_{B^c(x)\cap \{b> 0\}}|\overline{ K}_F(x,y)| |f(y)|dy&\lesssim \int_{B^c(x)}\frac{e^{-(1-\epsilon)u_0}}{t_0^{d/2}}e^{\epsilon\left(|x|^2-|y|^2\right)}|f(y)|dy\nonumber \\ & =e^{\frac{|x|^2}{p(x)}}\int_{B^c(x)}\frac{e^{-(1-\epsilon)u_0}e^{{\frac{|y|^2}{p(y)}}-\frac{|x|^2}{p(x)}}}{t_0^{d/2}}\nonumber \\ &\qquad \times  e^{\epsilon\left(|x|^2-|y|^2\right)} |f(y)|e^{-\frac{|y|^2}{p(y)}}dy,  
\end{align}
Since $p\in \mathcal{P}_{\gamma_d}^\infty(\mathbb R^d)$, from \cite{DS} we know that \begin{align*}\frac{e^{-(1-\epsilon)u_0}e^{{\frac{|y|^2}{p(y)}}-\frac{|x|^2}{p(x)}}}{t_0^{d/2}}&\lesssim \frac{e^{(|y|^2-|x|^2)\left(\frac{1}{p_\infty}-\frac{1-\epsilon}{2}\right)}}{t_0^{d/2}}e^{-\frac{1-\epsilon}{2}|x+y||x-y|}\\ & \lesssim |x+y|^d e^{(|y|^2-|x|^2)\left(\frac{1}{p_\infty}-\frac{1-\epsilon}{2}\right)}e^{-\frac{1-\epsilon}{2}|x+y||x-y|}. \end{align*} 
Then, applying this to \eqref{eq: KF,b>0} and taking into account that $||y|^2-|x|^2|\le |x+y||x-y|,$ we obtain
\begin{align*}
\int_{B^c(x)\cap \{b> 0\}}|\overline{ K}_F(x,y)| |f(y)|dy &\lesssim e^{\frac{|x|^2}{p(x)}}\int_{B^c(x)}|x+y|^d e^{(|y|^2-|x|^2)\left(\frac{1}{p_\infty}-\frac{1-3\epsilon}{2}\right)}  \\ &\qquad \times e^{-\frac{1-\epsilon}{2}|x+y||x-y|} |f(y)| e^{-\frac{|y|^2}{p(y)}}dy\\ &\lesssim e^{\frac{|x|^2}{p(x)}}\int_{B^c(x)}|x+y|^d e^{-\alpha_\infty|x+y||x-y|}|f(y)| e^{-\frac{|y|^2}{p(y)}}dy.
\end{align*}
Finally, we may choose $\epsilon$ in such a way that $\alpha_\infty=\frac{1-\epsilon}{2}-\left|\frac{1}{p_\infty}-\frac{1-3\epsilon}{2}\right|>0;$ for example we can take $0<\epsilon<\frac{1}{2p_\infty'}\wedge \frac1d.$

\bigskip

\section{Proof of main results}

In order to prove the $L^{p(\cdot)}$-boundedness of $\overline{R}_F$ in the Gaussian context, we will use the continuity properties of Calder\'on-Zygmund singular integrals on the Lebesgue setting.
	
 It is known that $p\in LH(\mathbb{R}^d)$ is sufficient for the boundedness on $L^{p(\cdot)}(\mathbb R^d)$ of singular integral operators (see \cite[Theorem~5.39]{CUF}). Here, it will be enough to consider singular integral operators with homogeneous kernels. That is, operators of the form
	\begin{equation}\label{eq: T}
	Tf(x)=\lim\limits_{\epsilon\rightarrow 0}\int_{\{|y|\ge \epsilon\}}\frac{\Omega(y')}{|y|^d}f(x-y)dy,
	\end{equation}
	for $f\in \mathcal S$ (the class of Schwartz functions), where $\Omega$ is defined on the unit sphere $S^{d-1}$, is integrable with zero average and $y'=y/|y|$. This kind of operators, and a wider class of singular integrals, are bounded on $L^p(\mathbb R^d)$ (see \cite{Duo, GCRF, Grafakos}). Moreover, the next result is valid on the variable setting.

	\begin{thm}[\cite{CUFN,CUF}]\label{thm: MHL-T}Let $p\in LH(\mathbb{R}^d)$ with $1<p^-\le p^+<\infty$. Then, the Hardy-Littlewood maximal operator $M_{\text{H-L}}$ and singular integrals with homogeneous kernels of the form \eqref{eq: T} are bounded on $L^{p(\cdot)}(\mathbb R^d)$.
	\end{thm}

We can now prove our main result.

\begin{thm}\label{thm: local+global}Let $p\in LH_0(\mathbb{R}^d)\cap\mathcal{P}_{\gamma_d}^\infty(\mathbb{R}^d)$, with $p^->1$. Then, there exists a positive constant $C$ such that
\[\|\overline{R}_F f\|_{p(\cdot),\gamma_d}\le C\|f\|_{p(\cdot),\gamma_d}\]
for every $f\in L^{p(\cdot)}(\gamma_d)$.
\end{thm}

\begin{proof}

Since for every $x\in \mathbb{R}^d,$ $|\overline{R}_F f(x)|\lesssim \boldsymbol{L}f(x)+\boldsymbol{G}f(x),$ the proof of Theorem \ref{thm: local+global} will follow from the boundedness of $\boldsymbol{L}$ and $\boldsymbol{G}$ on $ L^{p(\cdot)}(\gamma_d)$.

The boundedness of $\boldsymbol{L}$ follows the same lines as the proof of \cite[Theorem 3.3]{DS}, by means of \eqref{eq: localpart} and applying Theorem \ref{thm: MHL-T} since the operator $\overline{T}_F$ given in \eqref{eq: T-CZ} falls in its scope.

In order to prove that 
\[\|\boldsymbol{G}f\|_{p(\cdot),\gamma_d}\lesssim \|f\|_{p(\cdot), \gamma_d},\] we will carry out the same steps we used in \cite{DS} for the ``old'' Gaussian Riesz transforms, with a few changes. Let $f\in L^{p(\cdot)}(\gamma_d)$ such that $\|f\|_{p(\cdot),\gamma_d}=1$. We will prove that $\int_{\mathbb{R}^d}(\boldsymbol{G}f(x))^{p(x)}d\gamma_d(x)\lesssim 1,$ then by homogeneity the general boundedness will yield.

It is easy to prove that
\begin{equation}\label{eq: intG<=1}
\int_{\mathbb{R}^d}e^{\epsilon p(x)|x|^2}\left(\int_{\mathbb{R}^d} |f(y)|^{p^-}d\gamma_d(y)\right)^{\frac{p(x)}{p^-}}d\gamma_d(x)\lesssim 1.
\end{equation}
Indeed, since $p\in \mathcal{P}_{\gamma_d}^\infty (\mathbb{R}^d),$ then $p(x)\le p_\infty +\frac{C_{\gamma_d}}{|x|^2},$ and thus
\[e^{\epsilon p(x)|x|^2}\le e^{\epsilon p_\infty |x|^2}e^{\epsilon C_{\gamma_d}}.\]
Also 
\[\int_{\mathbb{R}^d} |f(y)|^{p^-}d\gamma_d(y)\le 1+\int_{|f|>1}|f(y)|^{p^-}d\gamma_d(y)\le 1+\int_{\mathbb{R}^d}|f(y)|^{p(y)}d\gamma_d(y)\le 2,\] 
so we have
\begin{align*}\int_{\mathbb{R}^d}e^{\epsilon p(x)|x|^2}\left(\int_{\mathbb{R}^d} |f(y)|^{p^-}d\gamma_d(y)\right)^{\frac{p(x)}{p^-}}d\gamma_d(x)& \lesssim \int_{\mathbb{R}^d}e^{-(1-\epsilon p_\infty)|x|^2}2^{p(x)}dx\nonumber \\ &\lesssim 2^{p^+} \int_{\mathbb{R}^d}e^{-(1-\epsilon p_\infty)|x|^2}dx\nonumber 
\end{align*}
and this last integral is finite if we take $0<\epsilon<\frac{1}{p_\infty}.$ Thus, we obtain \eqref{eq: intG<=1} by choosing $0<\epsilon<\min\{\frac{1}{p_\infty},\frac{1}{2p'_\infty},\frac1d\}$. 

On the other hand, it can be proved that \begin{equation*}
D:=\sup_{x\in \mathbb{R}^d}\int_{B^c(x)}P(x,y)|f(y)| e^{-\frac{|y|^2}{p(y)}} dy<\infty,
\end{equation*}
see \cite{DS}. Then, if we set $g(y):=|f(y)| e^{-\frac{|y|^2}{p(y)}}=g_1(y)+g_2(y)$ with $g_1=g\chi_{\{g>1\}},$ as it was done in \cite{DS}, taking into account that $0\le \frac{1}{D}\int_{B^c(x)}P(x,y)g_1(y)dy\le 1,$ using Lemma \ref{lem: changep} conveniently, and realizing that both  $0\le \frac{1}{D}\int_{B^c(x)}P(x,y)g_2(y)dy\le 1$ and $0\le g_2(y)\le 1,$ we have
\begin{align*}
\int_{\mathbb{R}^d}& e^{|x|^2}\left(\int_{B^c(x)}P(x,y)g(y)dy\right)^{p(x)}d\gamma_d(x) \\ & \lesssim \int_{\mathbb{R}^d}\left(\frac{1}{D}\int_{B^c(x)}P(x,y)g_1(y)dy\right)^{p(x)} dx\\ & \quad +\int_{\mathbb{R}^d}\left(\frac{1}{D}\int_{B^c(x)}P(x,y)g_2(y)dy\right)^{p(x)} dx \\ &\lesssim \int_{\mathbb{R}^d}\left(\left(\int_{B^c(x)}P(x,y)g_1(y)dy\right)^{p^-}+\left(\int_{B^c(x)}P(x,y)g_2(y)dy\right)^{p_\infty}\right)dx+1\\ & \lesssim \int_{\mathbb{R}^d}(g_1(y)^{p^-}+g_2(y)^{p_\infty}) dy+1 \\ &\lesssim \int_{\mathbb{R}^d}|f(y)|^{p(y)}d\gamma_d(y)+\int_{\mathbb{R}^d}g_2(y)^{p(y)}dy+1\lesssim 1.
\end{align*}
Thus, $\|\boldsymbol{G}f\|_{p(\cdot),\gamma_d}\lesssim 1.$ And this ends the proof of Theorem \ref{thm: local+global}.
\end{proof}

\section{The non-centered Gaussian maximal function}\label{sec: maximal-variable}

Now let us introduce here the non-centered maximal function associated to the non-standard Gaussian measure, i.e., \[\mathcal{M}_{\gamma_d} f(x)=\sup\limits_{B\ni x} \frac{1}{\gamma_d(B)}\int_B |f(y)| d\gamma_d(y),\]
	where the supremum is taken over every ball $B$ of $\mathbb R^d$ containing $x$. 
	
	It is known that this maximal function  is a bounded operator on $L^p(\gamma_d)$ for $1<p\le \infty$ (see \cite{FSSU}) and it is not weak-type $(1,1)$ (see \cite{Sjogren}).
	
	Under certain conditions on the exponent $p(\cdot)$, we are going to prove the boundedness of $\mathcal{M}_{\gamma_d}$ on $L^{p(\cdot)}(\gamma_d)$ whenever $1<p^-\le p^+<\infty$. 
	
	As a matter of fact we will prove a result that contains this one where we extend the space $\mathbb{R}^d$ to a metric space $X$ in which a positive $\sigma$-finite measure $\mu$ is defined such that $0<\mu(B)<\infty$ for all ball $B$ in $X.$ So the non-centered maximal function associated to $\mu$ is \[\mathcal{M}_\mu f(x)=\sup\limits_{B\ni x} \frac{1}{\mu(B)}\int_B |f(y)| d\mu(y),\]
	where the supremum is taken over every ball $B$ of $X$ containing $x$. 
	
	In  \cite[Theorem~1.7]{AHH} they gave a proof of this result for the centered maximal function  \[\mathcal{M}_\mu^c f(x)=\sup\limits_{r>0} \frac{1}{\mu(B(x,r))}\int_{B(x,r)} |f(y)| d\mu(y).\] In their technique of proof, they used, among other things, that this centered maximal function is weak-type $(1,1)$ which in the case of the non-centered maximal function this statement need not be true (see \cite{Sjogren}).
	
For our purposes, $\mathcal{M}_\mu$ will be defined pointwise for $\mu$-a.e. $x\in X.$ Since we are dealing with averages over balls, and with pointwise estimates for them (see Proposition \ref{prop: Jensenmaximal}) we are going to assume, when calculating the maximal function, that the supremum over all balls will coincide with the supremum over a collection of countable balls that cover all of $X.$ That is, there exists a countable family of balls $\mathcal{F}$ such that $\bigcup \mathcal{F}=X$ and
\[\mathcal{M}_\mu f(x)=\sup_{B\in \mathcal{F}, B\ni x}\frac{1}{\mu(B)}\int_B|f(y)|\, d\mu(y).\]
From now on, we are going to consider just those balls belonging to $\mathcal{F}, $ without mentioning it. Let us remark that this is true for the case we are concerned, that is, when $\mu=\gamma_d.$

We give some notation. We will denote by $\mathcal{P}(X,d,\mu)$ the set of bounded exponents over the metric space $(X,d)$ with respect to a positive $\sigma$-finite measure $\mu$. It can be proved that the variable Lebesgue space $L^{p(\cdot)}(\mu)$ is a Banach space (see \cite[Lemma 3.1]{HHP}). We will still denote, as before, by $\varrho_{p(\cdot),\mu}$ and $\|\cdot\|_{p(\cdot),\mu}$ the modular and the norm on $L^{p(\cdot)}(\mu)$, respectively.
	
\begin{defn}
		Let $(X,d)$ be a metric space in which a positive $\sigma$-finite measure $\mu$ is defined such that for every ball $B$ in $X$, $0<\mu(B)<+\infty.$ We say that a $\mu$-measurable function $p:X\to [1,+\infty)$ belongs to $\mathcal{P}_\mu(X)$ if there exists a constant $c_\mu$ with $0<c_\mu<1$ such that 
		\begin{equation}\label{eq: mu-ballcondition}
			\mu(B)^{p_B^+ - p_B^-}\ge c_\mu,
		\end{equation}
		for all ball $B$ in $X.$
	\end{defn}

The relationship between this measure $\mu$ and the exponent function $p$ expressed in \eqref{eq: mu-ballcondition} says a lot about the behaviour of $p$ both locally and its decay at infinity if we are dealing with an unbounded space $X$. For the case of the non-standard Gaussian measure $\gamma_d,$ we will get necessary and sufficient continuity conditions which $p$ must meet in order to hold inequality \eqref{eq: mu-ballcondition} true.

In this context we will prove the following theorem.
\begin{thm}\label{tmh: maximal}Let $\mu$ be a positive $\sigma$-finite measure on $X$ a metric space such that $0<\mu(B)<\infty$ for every ball $B$. Let $p\in \mathcal{P}_\mu(X)$ with $p^->1$. 
	
	If $\mu(X)=\infty$, we also assume that there exists $p_\infty \in [1,\infty)$ such that $1\in L^{s(\cdot)}(\mu)$, where $\frac{1}{s(x)}=\left|\frac{1}{p(x)}-\frac{1}{p_\infty}\right|$.
	
	Then, if $\mathcal{M}_\mu$ is bounded on $L^{p^-}(\mu)$, we have
	\[\|\mathcal{M}_\mu f\|_{p(\cdot)}\le K \|f\|_{p(\cdot)}\]
	for every $f\in L^{p(\cdot)}(\mu)$.
\end{thm}

\begin{rem} In the Euclidean setting, the assumption $1\in L^{s(\cdot)}(\mathbb R^d)$ is nothing but Nekvinda's integral condition on the exponent $p$ (see \cite{Nek}). This property is strictly weaker than $p\in LH_\infty(\mathbb R^d)$ (see for example \cite[Proposition 4.9]{CUF}) but also sufficient, together with the local log-H\"older condition $LH_0(\mathbb R^d)$, for the boundedness of $M_{H-L}$ on $L^{p(\cdot)}(\mathbb R^d)$ as proved by Nekvinda in the mentioned article \cite{Nek}.
\end{rem}

\begin{rem}
	The proof can be done as in \cite{AHH}. But in their proof the authors obtain the result for the centered maximal function. Indeed, they use that this maximal function is weak-type $(1,1)$ which for the non-centered one this claim need not be true as aforementioned, see \cite{Sjogren}. However, we extend this result, following closely their proof, to the non-centered maximal function.
	
	We should also note that in the lemmas and theorems their proof is based on, it is missing the phrase ``$\mu$-almost everywhere'' since the exponent function may not necessarily be continuous under the given conditions.
\end{rem}

We recall some auxiliary results given in \cite{AHH} for the sake of completeness, and we state them taking into account the previous remark.

Next lemma corresponds to \cite[Lemma A1]{AHH}. Here, we establish the precise constant $\beta\in (0,1)$ for our case, given that $p^+<\infty$. Indeed, $\beta=c_\mu$.

\begin{lem}\label{lem: A1}
Let $p\in \mathcal{P}_\mu (X).$ Then, 
\begin{equation*}
\left(c_\mu \left(\frac{\lambda}{\mu(B)}\right)^{\frac{1}{p^-}}\right)^{p(x)}\le \frac{\lambda}{\mu(B)},
\end{equation*}
for every $ \lambda\in [0,1],$ for $\mu$-a.e. $x\in B$ and for each ball $ B$ in $X.$
\end{lem}

The above condition yields the following estimate, that will lead to a pointwise inequality for $\mathcal{M}_\mu$.

\begin{lem}\label{lem: variableJensen}
Let $p\in \mathcal{P}_\mu (X)$ be given and define $q:X\times X\to [1,+\infty]$ as follows \[\frac{1}{q(x,y)}=\max\left\{\frac{1}{p(x)}-\frac{1}{p(y)},0\right\}.\] Then, for every $\gamma\in (0,1)$, there exists $\delta\in (0,1)$ such that 
\begin{equation}\label{eq: variableJensen}
\left(\delta\fint_B|f(y)|\, d\mu(y)\right)^{p(x)}\le \fint_B |f(y)|^{p(y)}\, d\mu(y)+\fint_B \gamma^{q(x,y)}\, d\mu(y),
\end{equation}
for every ball $B$ in $X,$ $\mu$-a.e. $x\in B,$ and $f\in L^{p(\cdot)}(\mu)$ with $\|f\|_{p(\cdot),\mu}\le \frac{1}{2}$. Here, we set $\gamma^\infty=0.$
\end{lem}

\begin{proof}

Taking into account the embedding result given in \cite[Theorem 3.3.11]{DHHR}, for $f\in L^{p(\cdot)}(\mu)$ there exist $f_0:=\max\{|f|-1,0\}\in L^{\frac{p(\cdot)}{p^-}}(\mu)$ and $f_1:=\min\{|f|,1\}\in L^\infty(\mu)$  such that $|f|=f_0+f_1$ and $\|f_0\|_{p(\cdot)/p^-,\mu}+\|f_1\|_\infty\le 2\|f\|_{p(\cdot),\mu}\le 1,$ since we shall assume $\|f\|_{p(\cdot),\mu}\le \frac12$.

 Let $B\subset X$ be a ball and $E_B\subset B$ with $\mu(E_B)=0$ such that for any $x\in B\setminus E_B$, $1\le p(x)<+\infty.$ Fix such an $x$.

Let $\beta\in (0,1)$ be the constant obtained in Lemma \ref{lem: A1}. We can also assume $\beta\le \gamma.$  Now we will call $g$ to either $f_0$ or $f_1,$ and split it into the sum of three functions:
\begin{align*}
g_1(y)&=g(y)\chi_{\{z\in B: |g(z)|>1\}}(y), \\
g_2(y)&=g(y)\chi_{\{z\in B:|g(z)|\le 1, p(z)\le p(x)\}}(y),\\
g_3(y)&=g(y)\chi_{\{z\in B:|g(z)|\le 1, p(z)> p(x)\}}(y).
\end{align*} 
Let us remark that $g\chi_B=g_1+g_2+g_3.$
 Let us also observe that $(f_1)_1\equiv 0.$ 

By the convexity of $t\mapsto t^{p(x)}$, 
\[\left(\frac{\beta}{3}\fint_Bg(y)\, d\mu(y)\right)^{p(x)}\le \frac{1}{3}\sum_{j=1}^3\left(\beta\fint_Bg_j(y)\, d\mu(y)\right)^{p(x)}=:\frac{1}{3}(I_1+I_2+I_3).\]
 Let us prove that \[I_j\le \fint_B g(y)^{p(y)}\, d\mu(y), \quad j=1,2\] and 
 \[I_3\le \fint_B g(y)^{p(y)}\, d\mu(y)+\fint_B \gamma^{q(x,y)}d\mu(y).\]

By applying H\"older's inequality and taking into account that $t\mapsto t^{p(x)}$ is a non-decresasing function, we get
\[I_1\le \left(\beta \left(\fint_B g_1(y)^{p_B^-}d\mu(y)\right)^{\frac{1}{p_B^-}}\right)^{p(x)}.\]
Since $g_1=0$ or $g_1>1$ and $p_B^- \le p(y)$ $\mu$-a.e. $y\in B$, we have $g_1^{p_B^-}(y)\le g_1^{p(y)}(y)$ $\mu$-a.e. $y\in B.$ Then 
\[I_1\le \left(\beta \left(\fint_B g_1(y)^{p(y)}d\mu(y)\right)^{\frac{1}{p_B^-}}\right)^{p(x)}.\]
Since $(f_1)_1=0$ then $I_1=0$ for $g=f_1$. On the other hand, since $\|f\|_{p(\cdot),\mu}\le \frac{1}{2} $ then $\int_B (f_0)_1^{p(y)}\, d\mu(y)\le 1.$ So by applying Lemma \ref{lem: A1} with $\lambda=\fint_Bg_1^{p(y)}\, d\mu(y),$ $0\le \lambda\le 1,$ we get \[I_1\le \fint_B g_1(y)^{p(y)}\, d\mu(y)\le \fint_B g(y)^{p(y)}\, d\mu(y).\]
Jensen’s inequality implies that
\[I_2 \le \fint_B
(\beta|g_2(y)|)^{p(x)}\, d\mu(y).\]
Since $\beta|g_2(y)| \le |g_2(y)| \le 1$ and $t^{p(x)} \le t^{p(y)}$ for all $t \in [0, 1]$ whenever $p(y) \le p(x),$ we
obtain that
\[I_2 \le \fint_B
(\beta |g_2(y)|)^{p(y)}\, d\mu(y) \le \fint_B (|g_2(y)|)^{p(y)}\, 
d\mu(y) \le \fint_B
|g(y)|^{p(y)}\, d\mu(y).\]
Finally, for $I_3$ we get with Jensen’s inequality that
\[I_3 \le  \fint_B
(\beta|g(y)|)^{p(x)} \chi_{\{y\in B: |g(y)|\le 1,p(y)>p(x)\}}\,  d\mu(y).\]
Now, Young’s inequality (see e.g. \cite[Lemma 3.2.15]{DHHR}), the definition of $q(x, y)$ and $\beta\le \gamma$,
give that
\begin{align*}I_3 & \le   \fint_B\left( \left( \frac{\beta |g(y)|}{\gamma}
\right)^{p(y)}+ \gamma^{q(x,y)} \right)
\chi_{\{y\in B: |g(y)|\le 1, p(y)>p(x)\}}\, d\mu(y) \\ 
&\le  \fint_B
|g(y)|^{p(y)}\, d\mu(y) +  \fint_B
\gamma^{q(x,y)}  d\mu(y). \end{align*}

This proves inequality \eqref{eq: variableJensen} for $f_0$ and $f_1$. To get that inequality for $f,$ taking into account that $t\mapsto t^{p(x)}$ is a convex function, we argue
\[\left(\frac{\beta}{6}\fint_B|f(y)|\, d\mu(y)\right)^{p(x)}\le \frac{1}{2}\left[\left(\frac{\beta}{3}\fint_B f_0(y)\, d\mu(y)\right)^{p(x)}+\left(\frac{\beta}{3}\fint_B f_1(y)\, d\mu(y)\right)^{p(x)}\right].\]
Then, by applying the lemma for $f_0$ and $f_1$ and taking into account that $f_j\le |f|$ for $j=0,1$, we prove this lemma for $f$ as well by choosing $\delta=\frac{\beta}{6}.$
\end{proof}

The following lemma is immediate, and will be used in the proof of Theorem \ref{tmh: maximal} for the case $\mu(X)=\infty.$

\begin{lem}\label{lem: A4}

Let $q$ be the exponent defined in Lemma \ref{lem: variableJensen} and define a new exponent $s:X\to [1,+\infty]$ by \[\frac{1}{s(x)}=\left|\frac{1}{p(x)}-\frac{1}{p_\infty}\right|,\] for some constant $p_\infty\in [1,\infty)$. Then
\[t^{q(x,y)}\le t^{\frac{s(x)}{2}}+t^{\frac{s(y)}{2}}\] for every $t\in [0,1].$
\end{lem}

By combining Lemmas \ref{lem: variableJensen} and \ref{lem: A4}, the following result can be deduced.

\begin{thm}\label{thm: variableJensen2}
Let $p\in \mathcal{P}_\mu(X).$ Then for every $\gamma\in(0,1)$ there exists $\delta\in (0,1)$ such that 
\begin{align*}
\left(\delta\fint_B|f(y)|\, d\mu(y)\right)^{p(x)}\le \fint_B |f(y)|^{p(y)}\, d\mu(y)+\fint_B \left(\gamma^{\frac{s(x)}{2}}+\gamma^{\frac{s(y)}{2}}\right)\, d\mu(y),
\end{align*}
for every ball $B$ in $X,$ $\mu$-a.e. $x\in B,$ $f\in L^{p(\cdot)}(\mu)$ with $\|f\|_{p(\cdot),\mu}\le \frac{1}{2},$ being $s(\cdot)$ as in Lemma \ref{lem: A4}.
\end{thm}

\begin{prop}\label{prop: Jensenmaximal}
Let $p\in \mathcal{P}_\mu(X).$ Then for every $\gamma\in (0,1)$ there exists $\delta\in (0,1)$ such that
\begin{equation}\label{eq: Jensenmaximal}(\delta \mathcal{M}_\mu f(x))^{p(x)}\le \mathcal{M}_{\mu}\left(|f|^{p(\cdot)}\right)(x)+2\mathcal{M}_\mu \left(\gamma^{\frac{s(\cdot)}{2}}\right)(x),\end{equation}
for all $f\in L^{p(\cdot)}(\mu)$ with $\|f\|_{p(\cdot),\mu}\le \frac{1}{2},$ $\mu$-a.e. $x\in X,$ being $s(\cdot)$ as in Lemma \ref{lem: A4}.
\end{prop}

\begin{proof} 
From Theorem \ref{thm: variableJensen2} taking the supremum over all balls $B \in \mathcal{F}$ and using that
$\gamma^{\frac{s(x)}{2}}\le \mathcal{M}_\mu\left(\gamma^{\frac{s(\cdot)}{2}}\right)(x)$ for $\mu$-a.e. $x\in X,$ we get inequality \eqref{eq: Jensenmaximal}.
\end{proof}

Finally we prove the main result.

\begin{proof}[Proof of Theorem \ref{tmh: maximal}]
Since $p\in \mathcal{P}_\mu(X)$ with constant $c_\mu$, then the exponent $q:=\frac{p(\cdot)}{p^-}\in \mathcal{P}_\mu(X)$ with constant $c_\mu^{\frac{1}{p^-}}$ and $q^-=1.$ Let us take $f\in L^{p(\cdot)}(\mu)$ with $\|f\|_{p(\cdot),\mu}\le \frac{1}{2}.$

By applying inequality \eqref{eq: Jensenmaximal} from Proposition \ref{prop: Jensenmaximal} for $q$ we get
\begin{align*}
\left(\delta \mathcal{M}_\mu f(x)\right)^{p(x)}&=\left(\left(\delta \mathcal{M}_\mu f(x)\right)^{q(x)}\right)^{p^-}\\
&\le \left(\mathcal{M}_\mu (|f|^{q(\cdot)})(x)+2\mathcal{M}_{\mu}\left(\gamma^{\frac{s(\cdot)}{2}}\right)(x)\right)^{p^-}\\
&\lesssim \left(\mathcal{M}_\mu(|f|^{q(\cdot)} )(x)\right)^{p^-}+\left(\mathcal{M}_\mu\left(\gamma^{\frac{s(\cdot)}{2}} \right)(x)\right)^{p^-}. 
\end{align*}
Integrating over $X$ yields
\[\varrho_{p(\cdot),\mu}(\delta \mathcal{M}_\mu f)\lesssim \left\|\mathcal{M}_\mu\left(|f|^{q(\cdot)} \right)\right\|_{p^-,\mu}^{p^-}+\left\|\mathcal{M}_\mu\left(\gamma^{\frac{s(\cdot)}{2}} \right)\right\|_{p^-,\mu}^{p^-}.\]

If $\mu(X)<+\infty,$ we use that the maximal $\mathcal{M}_\mu$ is bounded on both $L^{p^-}(\mu)$ and $L^\infty(\mu)$ taking into account that $|f|^{q(\cdot)} \in L^{p^-}(\mu)$ and $\gamma^{\frac{s(\cdot)}{2}}\in L^\infty(\mu)$ for every $ \gamma\in (0,1)$ with $\|\gamma^{\frac{s(\cdot)}{2}}\|_{\infty,\mu} \le 1.$ Thus
\[\varrho_{p(\cdot),\mu}(\delta \mathcal{M}_\mu f)\lesssim \left\||f|^{q(\cdot)}\right\|_{p^-,\mu}^{p^-}+\mu(X)\lesssim 1+\mu(X)<+\infty.\]

If $\mu(X)=+\infty,$ since $1\in L^{s(\cdot)}(\mu)$ then there exists $\lambda>1$ such that 
\[\int_{X}\left(\frac{1}{\lambda}\right)^{s(y)}\, d\mu(y)<+\infty.\]
By taking $\gamma=\lambda^{-2}\in (0,1)$ we have that $\gamma^{\frac{s(\cdot)}{2}}\in L^1(\mu )\cap L^\infty (\mu)$ and hence $\gamma^{\frac{s(\cdot)}{2}}\in L^{p^-}(\mu).$ And since $\mathcal{M}_{\mu}$ is bounded on $L^{p^-}(\mu)$ we have
\[\varrho_{p(\cdot),\mu}(\delta \mathcal{M}_\mu f)\lesssim \left\||f|^{q(\cdot)}\right\|_{p^-,\mu}^{p^-}+\left\|\gamma^{\frac{s(\cdot)}{2}}\right\|_{p^-,\mu}^{p^-}\lesssim 1+\left\|\gamma^{\frac{s(\cdot)}{2}}\right\|_{p^-,\mu}^{p^-}<+\infty.\]
And with this we end the proof of this theorem.
\end{proof}

\bigskip

We are now interested in giving sufficient pointwise conditions on $p(\cdot)$ such that $p\in \mathcal{P}_{\gamma_d}(\mathbb R^d)$ holds. 

Condition $p\in \mathcal{P}_{\mu}(X)$ is a generalization of Diening's geometric condition \eqref{eq: Diening-geom} when $\mu$ is the Lebesgue measure and $(X,d)=(\mathbb R^d,|\cdot|)$. However, it is not necessarily true that this is equivalent to the local log-H\"older condition $LH_0(\mathbb R^d)$ for every measure, see \cite[Lemma 3.6]{HHP}. 

From now on, for a given ball $B$ of radius $r_B>0$, we denote by $q_B$ the point in the closure of $B$ whose distance to the origin is minimal, i.e., $q_B\in \overline{B}$ and $|q_B|=\text{dist}(0,B)$.

The next lemma is technical and although it can be found as a partial result in the proof of \cite[Lemma~1]{FSSU} we are including it here for the sake of completeness. 

\begin{lem}[\cite{FSSU}]\label{lem: lowerboundgamma}Let $B$ be a ball of $\mathbb{R}^d$ of radius $r_B>0$, and let $q_B$ as defined before. If $|q_B|\ge 1$ and $r_B\ge  1/(4|q_B|)$, then
  \begin{equation*}
    \gamma_d(B)\ge C \frac{e^{-|q_B|^2}}{|q_B|}\left(1\wedge \left(\frac{r_B}{|q_B|}\right)^{\frac{d-1}{2}}\right),
  \end{equation*}
where $C$ does not depend on $B$.
\end{lem}
\begin{proof}
Consider the hyperplane orthogonal to $q_B$ whose distance from
the origin is $|q_B|+ t$, with $\frac{1}{2|q_B|}< t < \frac{1}{|q_B|}.$  Its intersection with $B$ is a $(d-1)$-dimensional ball whose radius is at least $ C\sqrt{r_B t}\ge \widetilde{C}\sqrt{r_B/|q_B|}.$ 
Integrating the Gaussian density first along this $(d-1)$-dimensional  ball and then in $t$, we get
\[\gamma_d(B)\ge \int_{1/(2|q_B|)}^{1/|q_B|}e^{-(|q_B|+t)^2}\int_{|v|< \widetilde{C}\sqrt{r_B/|q_B|}} e^{-|v|^2}dv dt\]
where $v$ is a $(d-1)$-dimensional variable. The inner integral here is at least
$C \left(1\wedge (r_B/|q_B|)^{(d-1)/2)}\right)$, and $e^{-(|q_B|+t)^2 }\ge C e^{-|q|^2}$ for these $t$. Therefore
\[\gamma_d(B)\ge C \frac{e^{-|q_B|^2}}{|q_B|}\left(1\wedge \left(\frac{r_B}{|q_B|}\right)^{\frac{d-1}{2}}\right).\qedhere\]
\end{proof}
Now we are in position to give sufficient conditions for the validity of $p\in \mathcal{P}_{\gamma_d}(\mathbb R^d)$ and, consequently, for the boundedness of $\mathcal M_{\gamma_d}$.

\begin{lem}Let $p\in  LH_0(\mathbb R^d)$ be given and assume that there exists a constant $C_{\gamma_d}$ such that 
\begin{equation}\label{eq: maxdifp}
	p_B^+-p_B^-\le \frac{C_{\gamma_d}}{|q_B|^2}
\end{equation}	
for every ball $B$. Then $p\in \mathcal{P}_{\gamma_d}(\mathbb R^d)$.
\end{lem}

\begin{proof}
	Let $B$ be a ball of center $c_B$ and radius $r_B>0.$ Then, there exists $0<c<1$ such that
	\begin{equation}\label{eq: lowerboundsgammaB}
		\left\{\begin{array}{lcl}
			\gamma_d(B)\ge c\, e^{-|q_B|^2}|B|&\text{if}&  r_B\le 1\wedge \frac{1}{|q_B|}\\
			\gamma_d(B)\ge c &\text{if}&|q_B|<1 \text{ and } r_B>1\\
			\gamma_d(B)\ge c\,  e^{-(d+1)|q_B|^2}&\text{if}&|q_B|\ge 1   \text{ and } r_B>\frac{1}{|q_B|}
		\end{array}\right. .
	\end{equation}
	Indeed, if $ r_B\le 1\wedge \frac{1}{|q_B|}$ and taking into account that for $y\in B,$ $|y|<|q_B|+2r_B,$ then \[\gamma_d(B)=\frac{1}{\pi^{d/2}}\int_B e^{-|y|^2}\, dy\ge \frac{1}{\pi^{d/2}} e^{-(|q_B|+2r_B)^2}|B|\ge \frac{e^{-8}}{\pi^{d/2}}e^{-|q_B|^2}|B|.\]
	If we are in the case $|q_B|<1$ and $r_B> 1,$ first assume that $|c_B|\le 1.$ Then $B(c_B,1)\subset B= B(c_B,r_B)$ and, thus, $\gamma_d(B)\ge \gamma_d(B(c_B,1))$. In this case, we also have $|y|\le 2$ for every $y\in B(c_B,1)$,  hence $\gamma_d(B)\ge e^{-4}(\omega_{d}/\pi^{d/2}).$ On the other hand, if $|c_B|>1,$ then $B\left(q_B+\frac{1}{|c_B|}c_B,1\right)\subset B,$ and for every $y\in B\left(q_B+\frac{1}{|c_B|}c_B,1\right)$, $|y|\le 3$. Hence, $\gamma_d(B)\ge e^{-9}(\omega_d/\pi^{d/2}).$
	
	Now  let us consider the case $|q_B|\ge 1$ and $r_B>\frac{1}{|q_B|}$. Then, from Lemma \ref{lem: lowerboundgamma}, we know that
	\begin{align*}\gamma_d(B)&\ge c\, \frac{e^{-|q_B|^2}}{|q_B|}\left(1\wedge\left(\frac{r_B}{|q_B|}\right)^{\frac{d-1}{2}}\right) \\ &\ge c\,  \frac{e^{-|q_B|^2}}{|q_B|}\left(1\wedge\left(\frac{1}{|q_B|^2}\right)^{\frac{d-1}{2}}\right) \\ &=\ c\, \frac{e^{-|q_B|^2}}{|q_B|^d}=c\, e^{-(|q_B|^2+d\log |q_B|)}\\ &\ge c\, e^{-(d+1)|q_B|^2}.
	\end{align*} 
	This finishes the proof of \eqref{eq: lowerboundsgammaB}. 
	
	Now, taking into account these estimates and condition \eqref{eq: maxdifp}, it can be easily seen that there exists a constant $c$ independent of $B$ such that \[\gamma_d(B)^{p^+_B-p^-_B}\ge c\] for every ball $B$ in $\mathbb{R}^d.$ The assumption $p\in LH_0(\mathbb R^d)$ is needed to estimate the case $r_B\le 1\wedge \frac{1}{|q_B|}$.
\end{proof}

Condition \eqref{eq: maxdifp} is actually equivalent to $\mathcal{P}_{\gamma_d}^\infty$, the condition considered at previous sections for the Gaussian Riesz transforms. We have the following lemma.

\begin{lem}Let $p\in \mathcal{P}(\mathbb R^d,\gamma_d)$. The following conditions are equivalent.
	\begin{enumerate}[label=(\roman*)]
		\item $p$ verifies \eqref{eq: maxdifp}; 
		\item $p\in \mathcal{P}_{\gamma_d}^\infty$, that is, for some constant $p_\infty\in [1,\infty)$, there exists $C_{\gamma_d}>0$ such that
		\begin{equation}\label{eq: Pinfty}
			|p(x)-p_\infty|\le \frac{C_{\gamma_d}}{|x|^2}, \quad \forall\, x\in \mathbb R^d\setminus \{(0,\dots, 0)\};
		\end{equation}
	    \item $p$ satisfies the inequality	\begin{equation}\label{eq: infdecay}
	    	|p(x)-p(y)|\le \frac{\widetilde{C_{\gamma_d}}}{|x|^2}, \quad \forall |y|\ge |x|,
	    \end{equation}
	    for some $\widetilde{C_{\gamma_d}}>0$.
	\end{enumerate}
\end{lem}

\begin{proof}It is easy to see that condition $p\in \mathcal{P}_{\gamma_d}^\infty$ is equivalent to the condition \eqref{eq: infdecay} where $p_\infty$ happens to be the limit of $p(x)$ as $|x|\rightarrow\infty $, uniformly in all directions. 

With this in mind, we will then prove \eqref{eq: maxdifp} $\Rightarrow$ \eqref{eq: infdecay} and \eqref{eq: Pinfty} $\Rightarrow$ \eqref{eq: maxdifp}. 

Assume \eqref{eq: maxdifp} holds. Take $|x|>\sqrt{2}$ (for $|x|\le \sqrt{2}$ the proof is immediate). Let us define $q_x=\left(|x|-\frac{1}{|x|}\right)\frac{x}{|x|}$ and the hyperspace \[H_0=\{z\in \mathbb{R}^d:x\cdot (y-x)\ge 0\}.\]
		For $y\in H_0$ we have $x\cdot y\ge |x|^2$ and it is easy to check that $|y|\ge |x|.$ Now we can choose a ball $B$ with $q_B=q_x$ and $x,y\in B.$ Indeed, $B=B(c_B,r_B)$ is chosen in such a way that $c_B=\lambda x$ for some $\lambda >1$ and radius $r_B=|q_x-c_B|=(\lambda -1)|x|+\frac{1}{|x|}.$ It is immediate that $|x-c_B|<r_B.$ The parameter $\lambda$ will be chosen greater than $1$ and depending on $x$ and $y$ subject to the condition $|y-c_B|<r_B.$ That is, $\lambda >\frac{2+|y|^2-|x|^2+\frac{1}{|x|^2}}{2(1+x\cdot y-|x|^2)}.$ Then, taking into account \eqref{eq: maxdifp} we have
		\begin{equation*}
			|p(y)-p(x)|\le p_B^+-p_B^-\le \frac{C_{\gamma_d}}{|q_x|^2}\le \frac{4C_{\gamma_d}}{|x|^2},
		\end{equation*}
		since $|q_x|\ge \frac{|x|}{2}.$ Thus,   $|p(y)-p(x)|\le \frac{4C_{\gamma_d}}{|x|^2}, $ for every $y\in H_0.$
		
		Now let us fix an angle $\theta \in (-\pi, \pi)$ and consider $\rho_\theta$ a rotation of an angle $\theta$ about the origin. Let us call $q_\theta=\rho_\theta q_x$ and $x_\theta=\rho_\theta x$ and define \[H_\theta=\{z\in \mathbb{R}^d: x_\theta \cdot (z-x_\theta)\ge 0\}.\] Since the module of a vector in $\mathbb{R}^d$ is invariant under rotations, we have $|q_\theta|=|q_x|\ge \frac{|x|}{2}$ and $|x_\theta|=|x|.$ Now we apply the same procedure as before and we get that \[|p(y)-p(x_\theta)|\le \frac{4C_{\gamma_d}}{|x|^2}\] for every $y\in H_\theta.$
		
		Let us remark that $H_\theta \cap H_0\ne \emptyset$ if and only if $-\pi <\theta<\pi.$
		
		For $y\in \mathbb{R}^d$ such that $|y|\ge |x|,$ let $\theta$ be the angle between $x$ and $y$ such that $|\theta|<\pi,$ then $y\in H_\theta.$ Let $z\in H_\theta \cap H_0,$ then 
		\[|p(y)-p(x)|\le |p(y)-p(x_\theta)|+|p(x_\theta)-p(z)|+|p(z)-p(x)|\le \frac{12 C_{\gamma_d}}{|x|^2}.\]
		We have proved \eqref{eq: infdecay} for $y\notin (B(0,|x|)\cup \{\alpha x:\alpha\le -1\}).$ Since $\{\alpha x:\alpha\le -1\}\subset H_{\pi}$ then for $0<\theta<\pi,$ we have $H_\pi\cap H_\theta\ne \emptyset$ 
		and proceeding as before we get the estimate $(16C_{\gamma_d})/|x|^2$ for $y\in \{\alpha x:\alpha\le -1\}.$ This ends the proof of \eqref{eq: infdecay}.
		
Now, let us assume
\eqref{eq: Pinfty} holds. For $B$ a ball with center at $c_B$ and radius $r_B$ such that $q_B\ne 0,$ and $x\in B,$ we know that $|p(x)-p_\infty|\le \frac{C_{\gamma_d}}{|q_B|^2}$ and then $p_\infty -\frac{C_{\gamma_d}}{|q_B|}\le p_B^-\le p_B^+\le p_\infty+\frac{C_{\gamma_d}}{|q_B|^2}.$ From this we easily get condition \eqref{eq: maxdifp}.
\end{proof}

Combining the previous lemmas, we deduce the following fact.

\begin{cor}\label{cor: condRiesz=>max}Let $p\in  LH_0(\mathbb R^d)\cap P_{\gamma_d}^\infty(\mathbb R^d)$. Then, $p\in \mathcal{P}_{\gamma_d}(\mathbb R^d)$.
\end{cor}

Since $\mathcal{M}_{\gamma_d}$ is bounded on $L^{p^-}(\gamma_d)$ and $\gamma_d(\mathbb{R}^d)<\infty,$ we can apply Corollary \ref{cor: condRiesz=>max} and Theorem \ref{tmh: maximal} in order to get  the continuity of $\mathcal{M}_{\gamma_d}$ on variable Lebesgue spaces with respect to $\gamma_d$ under the same sufficient conditions obtained  for the boundedness of the Gaussian Riesz transforms.

\begin{thm}Let $p\in LH_0(\mathbb{R}^d)\cap P_{\gamma_d}^\infty(\mathbb R^d)$ with $p^->1$. Then, there exists a constant $K>0$ such that
	\[\|\mathcal{M}_{\gamma_d}f\|_{p(\cdot),\gamma_d}\le K \|f\|_{p(\cdot),\gamma_d}\]
	for every $f\in L^{p(\cdot)}(\gamma_d)$.
\end{thm}

%\section*{Acknowlegdements}
%We are greatly appreciative to the referee for his/her appropriated suggestions.

\bibliographystyle{abbrv}
%\bibliography{DS}

\end{document}